\pdfoutput=1
\documentclass[12pt,reqno]{amsart}

\usepackage{settings}

\title{$ (k, a) $-generalized Fourier transform with negative $ a $}
\author{Tatsuro Hikawa}
\address[Tatsuro Hikawa]{%
  Graduate School of Mathematical Sciences, 
  the University of Tokyo,
  3-8-1 Komaba, Meguro-ku, Tokyo, 153-8914, Japan%
}
\dedicatory{%
  Dedicated to Professor Toshiyuki Kobayashi\\
  with admiration for his profound and creative works%
}

\begin{document}

\begin{abstract}
  The $ (k, a) $-generalized Fourier transform $ \FT{k, a} $ introduced
  by Ben Saïd--Kobayashi--Ørsted is a deformation family of the classical
  Fourier transform with a Dunkl parameter $ k $ and a parameter $ a > 0 $
  that interpolates minimal representations of two different simple Lie groups.
  In the present paper, we focus on the case $ a < 0 $. As a main result,
  we find a unitary transform that intertwines the known case $ a > 0 $
  and the new case $ a < 0 $.
\end{abstract}

\maketitle
\tableofcontents

\section{Introduction}

\subsection{Background}

A minimal representation is an infinite-dimensional unitary representation
with minimum Gelfand--Kirillov dimension. On the other hand, we can consider
a minimal representation as a manifestation of ``large symmetry'' of the space
acted by the group, and hence it is expected to control well global analysis on
the space. This is the idea of ``global analysis of minimal representations''
initiated by T. Kobayashi~\cite{Kob2011,Kob2013}, which caused a transition
from algebraic representation theory to analytic representation theory,
spawning an active research area that continues to the present day.

The Fourier transform in the classical harmonic analysis on $ \setR^N $ can
be interpreted as the unitary inversion operator appearing in the Schrödinger
model of the Weil representation, which is a unitary representation
of the metaplectic group $ \mathit{Mp}(N, \setR) $ on the Hilbert space
$ L^2(\setR^N) $ (see \cite{Fol1989} for more details),
and decomposes into two minimal representations.
This is a classical example of global analysis controlled well
by a minimal representation.

Kobayashi--Ørsted~\cite{KO2003a,KO2003b,KO2003c} analyzed the minimal
representation of the indefinite orthogonal group $ O(p, q) $
from various aspects, and in particular, they constructed a realization
of the minimal representation on $ L^2(C) $, or the Schrödinger model.
Here, $ C $ is a $ (p + q - 3) $-dimensional light cone in $ \setR^{p - 1, q - 1} $.
Based on the idea of ``global analysis of minimal representations'',
Kobayashi--Mano~\cite{KM2005,KM2007a,KM2007b,KM2011} built a new theory
of harmonic analysis by using the Schrödinger model $ L^2(C) $ of
the minimal representation of $ O(p, q) $ instead of that of
$ \mathit{Mp}(N, \setR) $. In particular, they introduced
the \emph{Fourier transform on the light cone}, which is a counterpart
of the classical Fourier transform. The special case $ (p, q) = (N + 1, 2) $,
where the Hilbert space $ L^2(C_+) $ ($ C_+ $ denotes the ``future part''
of the light cone) is isomorphic to $ L^2(\setR^N, \enorm{x}^{-1} \,dx) $,
is studied in \cite{KM2005,KM2007a}.

After that, Ben Saïd--Kobayashi--Ørsted~\cite{BKO2009,BKO2012} introduced
a family of $ \mathfrak{sl}_2 $-triples
\[
  \DiffH{k, a}, \quad \DiffEp{k, a}, \quad \DiffEm{k, a}
\]
of differential operators indexed by two parameters $ k $ and $ a $
(see \cref{ssec:sl2-triple}),
where $ k $ is a Dunkl parameter, and $ a $ is a deformation parameter.
By using this $ \mathfrak{sl}_2 $-triple, they introduced the
\emph{$ (k, a) $-generalized Laguerre semigroup}
$ \faml{\LS{k, a}(z)}{\RePart z \geq 0} $ (see \cref{ssec:ls}),
and its special value, the \emph{$ (k, a) $-generalized Fourier transform}
$ \FT{k, a} $ (see \cref{ssec:ft}).
The $ (k, a) $-generalized Fourier transform $ \FT{k, a} $ includes
some known transforms:
\begin{itemize}
  \item the classical Fourier transform $ \FT{} = \FT{0, 2} $,
  \item the Hankel transform $ \FT{0, 1} $,
        or the Fourier transform on the light cone when $ (p, q) = (N + 1, 2) $~\cite{KM2007a},
  \item the Dunkl transform $ \mscrD_k = \FT{k, 2} $~\cite{Dun1992}.
\end{itemize}
In particular, the parameter $ a $ interpolates continuously the minimal
representation of two simple Lie groups $ \mathit{Mp}(N, \setR) $
and $ O(N + 1, 2) $.

In recent years, the theory of the $ (k, a) $-generalized Fourier transform
has been actively researched from the perspectives of real analysis and
probability theory, beyond the framework of representation theory.

\subsection{Result of the paper}

In the present paper, we find a unitary transform that intertwines
the known case $ a > 0 $ and the new case $ a < 0 $
(\cref{thm:radial-intertwining,thm:total-intertwining}).
Although the $ \mathfrak{sl}_2 $-triple
$ (\DiffH{k, a}, \DiffEp{k, a}, \DiffEm{k, a}) $ is defined for any
$ a \in \setC \setminus \setenum{0} $, $ L^2 $-theory is only available
for $ a > 0 $ in the previous studies. This intertwining operator allows
us to extend the results of \cite{BKO2009,BKO2012} to the case $ a < 0 $.
In particular, we naturally extend the definition of the $ (k, a) $-generalized
Laguerre semigroup and Fourier transform to the case $ a < 0 $
(\cref{ssec:ls-with-negative-a,ssec:ft-with-negative-a}).

\subsection{Organization of the paper}

In \cref{sec:review}, we briefly review the theory of $ (k, a) $-generalized
Fourier transforms by \cite{BKO2009,BKO2012}. This section contains no new results.
In \cref{sec:negative-a}, we show a unitary transform that intertwines
the cases $ a > 0 $ and $ a < 0 $, which is used to extend the results
stated in \cref{sec:review} to the case $ a < 0 $.

\subsection{Notation}

\begin{itemize}
  \item $ \setN = \setenum{0, 1, 2, \dots} $, $ \setRp = \set{x \in \setR}{x > 0} $.
  \item We write $ \innprod{\blank}{\blank} $ for the Euclidean inner product,
        and $ \enorm{\blank} $ for the Euclidean norm.
  \item Function spaces, such as $ C^\infty $ spaces and $ L^2 $ spaces,
        are understood to consist of complex-valued functions.
\end{itemize}

\section{Review of the theory of \texorpdfstring{$ (k, a) $}{(k, a)}-generalized Fourier transforms}
\label{sec:review}

In this section, we briefly review the theory of $ (k, a) $-generalized Fourier
transform by \cite{BKO2009,BKO2012}. This section contains no new results.
Symbols generally follow from \cite{BKO2012}, but for convenience,
new or slightly modified symbols are used in a few places.

\subsection{Preliminaries on the Dunkl theory}

In this subsection, we explain some basic definitions and facts on the Dunkl
theory~\cite{Dun1988,Dun1989} as in \cite[Section~2]{BKO2012}.

For $ \alpha \in \setR^N \setminus \setenum{0} $, we write $ r_\alpha $ for
the orthogonal reflection with respect to the hyperplane $ (\setR \alpha)^\perp $,
namely,
\[
  r_\alpha(x) = x - \frac{2 \innprod{\alpha}{x}}{\enorm{\alpha}^2} \alpha
  \quad (x \in \setR^N).
\]
Let $ \mscrR $ be a reduced root system on $ \setR^N $ in the sense that
\begin{itemize}
  \item $ \mscrR $ is a finite subset of $ \setR^N \setminus \setenum{0} $,
  \item $ r_\alpha(\mscrR) = \mscrR $ for all $ \alpha \in \mscrR $, and
  \item $ \mscrR \cap \setR \alpha = \setenum{\alpha, -\alpha} $
        for all $ \alpha \in \mscrR $.
\end{itemize}
Note that we do not impose crystallographic conditions on roots
and do not require that $ \mscrR $ spans $ \setR^N $.

The subgroup $ \mfrakC $ of $ O(N) $ generated by all the reflections
$ r_\alpha $ is called the \emph{finite Coxeter group associated with $ \mscrR $}.
We say that a function $ \map{k}{\mscrR}{\setC} $ is a \emph{multiplicity function}
if it is constant on every $ \mfrakC $-orbit.
We usually write $ k_\alpha $ instead of $ k(\alpha) $.
We say that a multiplicity function $ k $ is \emph{non-negative}
if $ k_\alpha \geq 0 $ for all $ \alpha \in \mscrR $.
The \emph{index} of a multiplicity function $ k $ is defined as
\[
  \dunklindex{k}
  = \frac{1}{2} \sum_{\alpha \in \mscrR} k_\alpha
  = \sum_{\alpha \in \mscrR^+} k_\alpha,
\]
where $ \mscrR^+ $ is any positive system of $ \mscrR $.

For a (not necessarily non-negative) multiplicity function $ k $,
the \emph{Dunkl Laplacian} $ \Laplacian_k $
(see \cite{Dun1988} and \cite[1.1~Definiton]{Dun1989})
is defined by
\[
  \Laplacian_k f(x)
  = \Laplacian f(x)
    + \sum_{\alpha \in \mscrR^+} k_\alpha
    \Paren*{
      \frac{2 \innprod{\nabla f(x)}{\alpha}}{\innprod{\alpha}{x}}
      - \enorm{\alpha}^2 \frac{f(x) - f(r_\alpha(x))}{\innprod{\alpha}{x}^2}
    },
\]
where $ \Laplacian = \sum_{j = 1}^{N} (\Pdif{x_j})^2 $ is the classical Laplacian
and $ \nabla = (\Pdif{x_1}, \dots, \Pdif{x_N}) $ is the gradient operator.
When $ k = 0 $, $ \Laplacian_k $ reduces to the classical Laplacian $ \Laplacian $.

Let $ \mcalP^m(\setR^N) $ denote the space of homogeneous polynomials of degree $ m $.
We define the space of \emph{$ k $-harmonic polynomials of degree $ m $} as
\[
  \mcalH_k^m(\setR^N)
  = \set{p \in \mcalP^m(\setR^N)}{\Laplacian_k p = 0},
\]
and the space of \emph{$ k $-spherical harmonics of degree $ m $} as
\[
  \mcalH_k^m(S^{N - 1})
  = \set{\restr{p}{S^{N - 1}}}{p \in \mcalH_k^m(\setR^N)}.
\]
These spaces are $ \mfrakC $-stable. When $ k = 0 $, these are reduced to
the space $ \mcalH^m(\setR^N) $ of classical harmonic polynomials and
the space $ \mcalH^m(S^{N - 1}) $ of classical spherical harmonics, respectively,
which are $ O(N) $-stable.

We have the following Hilbert sum decomposition analogous to the classical
case ($ k = 0 $).

\begin{fact}[{\cite[pp.\,37--39]{Dun1988}}]\label{thm:decomposition}
  For a non-negative multiplicity function $ k $, we have the Hilbert sum
  decomposition
  \[
    L^2(S^{N - 1}, w_k(\omega) \,d\omega)
    = \sumoplus_{m \in \setN} \mcalH_k^m(S^{N - 1}),
  \]
  where the weight function $ w_k $ with respect to the standard measure
  $ d\omega $ on $ S^{N - 1} $ is defined by
  \[
    w_k(\omega)
    = \prod_{\alpha \in \mscrR^+} \abs{\innprod{\alpha}{\omega}}^{2k_\alpha}
    \quad (\omega \in S^{N - 1}).
  \]
\end{fact}

Henceforth, we fix a reduced root system $ \mscrR $ on $ \setR^N $.

\subsection{\texorpdfstring{$ \mathfrak{sl}_2 $}{sl2}-triple of differential-difference operators}
\label{ssec:sl2-triple}

For a (not necessarily non-negative) multiplicity function $ k $
and $ a \in \setC \setminus \setenum{0} $,
the differential-difference operators on $ \setR^N \setminus \setenum{0} $
introduced by \cite{BKO2009,BKO2012} are
\begin{align*}
  \DiffH{k, a}  &= \frac{N - 2 + 2 \dunklindex{k} + a}{a} + \frac{2}{a} E, \\
  \DiffEp{k, a} &= \frac{i}{a} \enorm{x}^a, \\
  \DiffEm{k, a} &= \frac{i}{a} \enorm{x}^{2 - a} \Laplacian_k,
\end{align*}
where $ E = \sum_{j = 1}^{N} x_j \Pdif{x_j} $ is the Euler operator on $ \setR^N $.
They form an $ \mathfrak{sl}_2 $-triple~\cite[Theorem~3.2]{BKO2012},
and hence yield a Lie algebra representation
\[
  \map{\rep{k, a}}{\mathfrak{sl}(2, \setR)}{\End_{\setC}(C^\infty(\setR^N \setminus \setenum{0}))}.
\]
We again write $ \rep{k, a} $ for its complexification.

For the above parameters $ k $, $ a $ and $ m \in \setN $, we consider
the following differential-difference operators on $ \setRp $:
\begin{align*}
  \DiffH{k, a}[m]   &= \frac{N - 2 + 2 \dunklindex{k} + a}{a} + \frac{2}{a} \vartheta, \\
  \DiffEp{k, a}[m] &= \frac{i}{a} r^a, \\
  \DiffEm{k, a}[m] &= \frac{i}{a} r^{-a} (\vartheta - m) (\vartheta + N - 2 + 2 \dunklindex{k} + m),
\end{align*}
where $ \vartheta = r \frac{d}{dr} $ is the Euler operator on $ \setRp $.
They are the radial parts of $ \DiffH{k, a} $, $ \DiffEp{k, a} $,
and $ \DiffEm{k, a} $ respectively in the following sense.

\begin{proposition}\label{thm:radial}
  Let $ k $ be a (not necessarily non-negative) multiplicity function,
  $ a \in \setC \setminus \setenum{0} $, and $ m \in \setN $.
  For $ p \in \mcalH_k^m(S^{N - 1}) $ and $ f \in C^\infty(\setRp) $, we have
  \begin{align*}
    \DiffH{k, a}  (p \otimes f) &= p \otimes \DiffH{k, a}[m]   f, \\
    \DiffEp{k, a} (p \otimes f) &= p \otimes \DiffEp{k, a}[m] f, \\
    \DiffEm{k, a} (p \otimes f) &= p \otimes \DiffEm{k, a}[m] f,
  \end{align*}
  where $ p \otimes f $ denotes the function $ r\omega \mapsto p(\omega) f(r) $
  on $ \setR^N \setminus \setenum{0} $.
\end{proposition}

\begin{proof}
  The first and second equations are obvious. The last equation follows from
  the second equation in \cite[Lemma~3.6]{BKO2012}.
\end{proof}

In particular, $ \DiffH{k, a}[m] $, $ \DiffEp{k, a}[m] $, and
$ \DiffEm{k, a}[m] $ also form an $ \mathfrak{sl}_2 $-triple,
and hence yield a Lie algebra representation
\[
  \map{\rep{k, a}[m]}{\mathfrak{sl}(2, \setR)}{\End_{\setC}(C^\infty(\setRp))}.
\]
We again write $ \rep{k, a}[m] $ for its complexification.

\subsection{\texorpdfstring{$ L^2 $}{L2}-theory for the \texorpdfstring{$ \mathfrak{sl}_2 $}{sl2}-triple}
\label{ssec:l2-theory}

For a non-negative multiplicity function $ k $ and $ a \in \setR $, we consider
a weight function
\begin{equation}
  w_{k, a}(x)
  = \enorm{x}^{a - 2} \prod_{\alpha \in \mscrR^+} \abs{\innprod{\alpha}{x}}^{2k_\alpha}
  \quad (x \in \setR^N \setminus \setenum{0}),
  \label{eq:weight}
\end{equation}
and the Hilbert space $ L^2(\setR^N, w_{k, a}(x) \,dx) $.
By the polar decomposition $ w_{k, a}(x) \,dx
= w_k(\omega) \,d\omega \otimes r^{N - 3 + 2 \dunklindex{k} + a} \,dr $ and
\cref{thm:decomposition}, we have the Hilbert sum decomposition
\begin{align}
  L^2(\setR^N, w_{k, a}(x) \,dx)
  &= L^2(S^{N - 1}, w_k(\omega) \,d\omega)
    \otimeshat L^2(\setRp, r^{N - 3 + 2 \dunklindex{k} + a} \,dr) \notag \\
  &= \sumoplus_{m \in \setN} \mcalH_k^m(S^{N - 1})
    \otimes L^2(\setRp, r^{N - 3 + 2 \dunklindex{k} + a} \,dr).
  \label{eq:decomposition}
\end{align}

Henceforth, for $ \lambda \in \setC $ and $ l \in \setN $,
we write the Laguerre polynomial as
\[
  L^{(\lambda)}_l(t)
  = \sum_{j = 0}^{l} \frac{(-1)^j}{j!} \binom{\lambda + l}{l - j} t^j
  = \sum_{j = 0}^{l} \frac{(-1)^j}{j! (l - j)!} \frac{\Gamma(\lambda + l + 1)}{\Gamma(\lambda + j + 1)} t^j
  \quad (t \in \setRp).
\]

\begin{proposition}[{\cite[Proposition~3.15]{BKO2012}}]\label{thm:onb}
  Let $ k $ be a non-negative multiplicity function, $ a > 0 $,
  and $ m \in \setN $, and set
  \[
    \lambda_{k, a, m} = \frac{N - 2 + 2 \dunklindex{k} + 2m}{a}.
  \]
  For $ l \in \setN $, we define a function $ f_{k, a, m; l} $ on $ \setRp $ by
  \[
    f_{k, a, m; l}(r)
    = \Paren*{\frac{2^{\lambda_{k, a, m} + 1} \Gamma(l + 1)}{a^{\lambda_{k, a, m}} \Gamma(\lambda_{k, a, m} + l + 1)}}^{1/2}
      r^m
      L^{(\lambda_{k, a, m})}_l \Paren*{\frac{2}{a} r^a}
      \exp \Paren*{-\frac{1}{a} r^a}
    \quad (r \in \setRp).
  \]
  If $ \lambda_{k, a, m} > -1 $, $ \faml{f_{k, a, m; l}}{l \in \setN} $ is an
  orthonormal basis for $ L^2(\setRp, r^{N - 3 + 2 \dunklindex{k} + a} \,dr) $.
  \qed
\end{proposition}

Let
\[
  \mbfh   = \begin{pmatrix} 1 & 0 \\ 0 & -1 \end{pmatrix}, \quad
  \mbfe^+ = \begin{pmatrix} 0 & 1 \\ 0 & 0 \end{pmatrix}, \quad
  \mbfe^- = \begin{pmatrix} 0 & 0 \\ 1 & 0 \end{pmatrix}
\]
be the standard basis for $ \mathfrak{sl}(2, \setR) $, and
\begin{align*}
  \mbfk
  &= \Ad\Paren*{\frac{1}{\sqrt{2}} \begin{pmatrix} -i & -1 \\ 1 & i \end{pmatrix}} \mbfh
  = \begin{pmatrix} 0 & -i \\ i & 0 \end{pmatrix}, \\
  \mbfn^+
  &= \Ad\Paren*{\frac{1}{\sqrt{2}} \begin{pmatrix} -i & -1 \\ 1 & i \end{pmatrix}} \mbfe^+
  = \frac{1}{2} \begin{pmatrix} i & -1 \\ -1 & -i \end{pmatrix}, \\
  \mbfn^-
  &= \Ad\Paren*{\frac{1}{\sqrt{2}} \begin{pmatrix} -i & -1 \\ 1 & i \end{pmatrix}} \mbfe^-
  = \frac{1}{2} \begin{pmatrix} -i & -1 \\ -1 & i \end{pmatrix}
\end{align*}
be the Cayley transforms of them.

\begin{proposition}\label{thm:action}
  Let $ k $ be a non-negative multiplicity function, $ a > 0 $, and $ m \in \setN $
  with $ \lambda_{k, a, m} = \frac{N - 2 + 2 \dunklindex{k} + 2m}{a} > -1 $.
  For $ l \in \setN $, we have
  \begin{align*}
    \rep{k, a}[m](\mbfk)   f_{k, a, m; l} &= (\lambda_{k, a, m} + 2l + 1) f_{k, a, m; l}, \\
    \rep{k, a}[m](\mbfn^+) f_{k, a, m; l} &= i \sqrt{(l + 1)(\lambda_{k, a, m} + l + 1)} \, f_{k, a, m; l + 1}, \\
    \rep{k, a}[m](\mbfn^-) f_{k, a, m; l} &= i \sqrt{l (\lambda_{k, a, m} + l)} \, f_{k, a, m; l - 1},
  \end{align*}
  where we regard $ f_{k, a, m; -1} = 0 $.
\end{proposition}

\begin{proof}
  By \cref{thm:radial}, the assertion is essentially the same as
  \cite[Theorem~3.21]{BKO2012}.
\end{proof}

In particular, we determine the spectra of the differential-difference
operators $ \rep{k, a}[m](\mbfk) $ and $ \rep{k, a}(\mbfk) $,
which play a fundamental role for the definitions of $ (k, a) $-generalized
Laguerre semigroups and Fourier transforms (\cref{ssec:ls,ssec:ft}).

\begin{corollary}\label{thm:radial-spectrum}
  Let $ k $ be a non-negative multiplicity function, $ a > 0 $, and $ m \in \setN $
  with $ \lambda_{k, a, m} = \frac{N - 2 + 2 \dunklindex{k} + 2m}{a} > -1 $.
  The differential-difference operator
  \[
    \rep{k, a}[m](\mbfk)
    = \frac{1}{a} (r^a - r^{-a} (\vartheta - m) (\vartheta + N - 2 + 2 \dunklindex{k} + m))
  \]
  restricted to $ W_{k, a}^{(m)} = \lspan_{\setC} \set{f_{k, a, m; l}}{l \in \setN} $
  is essentially self-adjoint on $ L^2(\setRp, r^{N - 3 + 2 \dunklindex{k} + a} \,dr) $
  and diagonalized by the orthonormal basis $ \faml{f_{k, a, m; l}}{l \in \setN} $.
  The discrete spectrum corresponding to $ f_{k, a, m; l} $ is
  $ \lambda_{k, a, m} + 2l + 1 $.
\end{corollary}

\begin{proof}
  It follows from \cref{thm:onb} and \cref{thm:action}.
\end{proof}

\begin{corollary}[{\cite[Corollary~3.22]{BKO2012}}]\label{thm:total-spectrum}
  Let $ k $ be a non-negative multiplicity function and $ a > 0 $
  with $ \lambda_{k, a, 0} = \frac{N - 2 + 2 \dunklindex{k}}{a} > -1 $.
  The differential-difference operator
  \[
    \rep{k, a}(\mbfk)
    = \frac{1}{a} (\enorm{x}^a - \enorm{x}^{2 - a} \Laplacian_k)
  \]
  restricted to $ W_{k, a} = \bigoplus_{m \in \setN} \mcalH_k^m(S^{N - 1}) \otimes W_{k, a}^{(m)} $
  is essentially self-adjoint on $ L^2(\setR^N, w_{k, a}(x) \,dx) $,
  has no continuous spectra, and has the set of discrete spectra
  \[
    \begin{cases}
      \set{\lambda_{k, a, m} + 2l + 1}{\text{$ m $, $ l \in \setN $}} & (N \geq 2) \\
      \set{\frac{2 \dunklindex{k} \pm 1}{a} + 2l + 1}{l \in \setN}    & (N = 1).
    \end{cases}
  \]
\end{corollary}

\begin{proof}
  It follows from \cref{thm:radial-spectrum} combined with the Hilbert sum
  decomposition \cref{eq:decomposition} and \cref{thm:radial}.
  (Note that $ \mcalH_k^m(S^0) \neq \zeromod $ only for $ m = 0 $, $ 1 $
  when $ N = 1 $.)
\end{proof}

\subsection{Lifting to a unitary representation \texorpdfstring{$ \Rep{k, a} $}{Ωk, a}}
\label{ssec:lifting}

In this subsection, we write $ \widetilde{\mathit{SL}}(2, \setR) $
for the universal covering Lie group of $ \mathit{SL}(2, \setR) $,
and $ \widetilde{\mathit{SO}}(2) $ for the connected Lie subgroup of
$ \widetilde{\mathit{SL}}(2, \setR) $ with Lie algebra
$ \mathfrak{so}(2) = \setR i\mbfk $.

To clearify the assertions below, we pin down some basic definitions and facts
about $ (\mathfrak{sl}(2, \setC), \widetilde{\mathit{SO}}(2)) $-modules as in \cite{BKO2012}.
Let $ (\varpi, V) $ be a $ (\mathfrak{sl}(2, \setC), \widetilde{\mathit{SO}}(2)) $-module.
We say that $ v \in V \setminus \setenum{0} $ is a \emph{lowest weight vector
of weight $ \mu \in \setC $} if
\[
  \varpi(\mbfk) v = \mu v
  \quad \text{and} \quad
  \varpi(\mbfn^-) v = 0.
\]
We say that $ (\varpi, V) $ is a \emph{lowest weight module of weight $ \mu $}
if $ V $ is generated by such $ v $.
It is known that, for each $ \lambda \in \setC $, there exists a unique
irreducible lowest weight module of weight $ \lambda + 1 $ (up to isomorphism),
for which we write $ \pi_{\widetilde{\mathit{SO}}(2)}(\lambda) $~\cite[p.\,28]{BKO2012}.
Moreover, for $ \lambda \geq -1 $, $ \pi_{\widetilde{\mathit{SO}}(2)}(\lambda) $
lifts to a unique irreducible unitary representation $ \pi(\lambda) $
of $ \widetilde{\mathit{SL}}(2, \setR) $,
which is infinite-dimensional except for the one-dimensional trivial representation
$ \pi(-1) $~\cite[Fact~3.27, 1), 2), 7)]{BKO2012}.

\begin{theorem}\label{thm:radial-lifting}
  Let $ k $ be a non-negative multiplicity function, $ a > 0 $, and $ m \in \setN $
  with $ \lambda_{k, a, m} = \frac{N - 2 + 2 \dunklindex{k} + 2m}{a} > -1 $.
  We set
  \[
    W_{k, a}^{(m)}
    = \lspan_{\setC} \set{f_{k, a, m; l}}{l \in \setN}.
  \]
  Then, $ (\rep{k, a}[m], W_{k, a}^{(m)}) $ equips a natural
  $ (\mathfrak{sl}(2, \setC), \widetilde{\mathit{SO}}(2)) $-module structure,
  and is isomorphic to $ \pi_{\widetilde{\mathit{SO}}(2)}(\lambda_{k, a, m}) $.
  Moreover, the $ (\mathfrak{sl}(2, \setC), \widetilde{\mathit{SO}}(2)) $-module
  $ (\rep{k, a}[m], W_{k, a}^{(m)}) $ lifts to a unique unitary representation
  $ \Rep{k, a}[m] $ of $ \widetilde{\mathit{SL}}(2, \setR) $
  on the Hilbert space $ L^2(\setRp, r^{N - 3 + 2 \dunklindex{k} + a} \,dr) $,
  which is unitarily equivalent to $ \pi(\lambda_{k, a, m}) $.
\end{theorem}

\begin{theorem}\label{thm:total-lifting}
  Let $ k $ be a non-negative multiplicity function and $ a > 0 $
  with $ \lambda_{k, a, 0} = \frac{N - 2 + 2 \dunklindex{k}}{a} > -1 $.
  We set
  \[
    W_{k, a}
    = \bigoplus_{m \in \setN} \mcalH_k^m(S^{N - 1}) \otimes W_{k, a}^{(m)}.
  \]
  Then, $ (\rep{k, a}, W_{k, a}) $ equips a natural
  $ \mfrakC \times (\mathfrak{sl}(2, \setC), \widetilde{\mathit{SO}}(2)) $-module
  structure, where $ \mfrakC $ denotes the Coxeter group.
  Moreover, the $ \mfrakC \times (\mathfrak{sl}(2, \setC), \widetilde{\mathit{SO}}(2)) $-module
  $ (\rep{k, a}, W_{k, a}) $ lifts to a unique unitary representation
  $ \Rep{k, a} $ of $ \mfrakC \times \widetilde{\mathit{SL}}(2, \setR) $
  on the Hilbert space $ L^2(\setR, w_{k, a}(x) \,dx) $,
  which decomposes as
  \[
    L^2(\setR^N, w_{k, a}(x) \,dx)
    = \sumoplus_{m \in \setN} \mcalH_k^m(S^{N - 1})
      \otimes L^2(\setRp, r^{N - 3 + 2 \dunklindex{k} + a} \,dr).
  \]
  Here, $ \mfrakC $ acts on $ \mcalH_k^m(S^{N - 1}) $,
  and $ \widetilde{\mathit{SL}}(2, \setR) $ acts on
  $ L^2(\setRp, r^{N - 3 + 2 \dunklindex{k} + a} \,dr) $
  via the unitary representation $ \Rep{k, a}[m] $.
\end{theorem}

\begin{proof}[Proof of \cref{thm:radial-lifting,thm:total-lifting}]
  \cref{thm:total-lifting} is nothing but \cite[Theorems~3.28, 3.30, 3.31]{BKO2012},
  in the proof of which \cref{thm:radial-lifting} is essentially proven by using
  \cite[Fact~3.27, 7)]{BKO2012}.
\end{proof}

\subsection{\texorpdfstring{$ (k, a) $}{(k, a)}-generalized Laguerre semigroup}
\label{ssec:ls}

Let $ k $ be a non-negative multiplicity function, $ a > 0 $, and $ m \in \setN $
with $ \lambda_{k, a, m} = \frac{N - 2 + 2 \dunklindex{k} + 2m}{a} > -1 $.
For $ z \in \setC $, we define
\[
  \LS{k, a}[m](z)
  = \exp(-z \rep{k, a}[m](\mbfk))
  = \exp\Paren*{\frac{z}{a} (r^{-a} (\vartheta - m)(\vartheta + N - 2 + 2 \dunklindex{k} + m) - r^a)}.
\]
By \cref{thm:radial-spectrum}, each operator $ \LS{k, a}[m](z) $
on $ L^2(\setRp, r^{N - 3 + 2 \dunklindex{k} + a} \,dr) $
is diagonalized by the orthonormal basis $ \faml{f_{k, a, m; l}}{l \in \setN} $
and satisfies
\begin{equation}
  \LS{k, a}[m](z) f_{k, a, m; l}
  = e^{-z (\lambda_{k, a, m} + 2l + 1)} f_{k, a, m; l}.
  \label{eq:diagonalized-ls}
\end{equation}
In particular, $ \LS{k, a}[m](z) $ is a Hilbert--Schmidt operator
when $ \RePart z > 0 $, and a unitary operator when $ \RePart z = 0 $.

Now, let $ k $ be a non-negative multiplicity function and $ a > 0 $
with $ \lambda_{k, a, 0} = \frac{N - 2 + 2 \dunklindex{k}}{a} > -1 $.
For $ z \in \setC $, we define
\[
  \LS{k, a}(z)
  = \exp(-z \rep{k, a}(\mbfk))
  = \exp\Paren*{\frac{z}{a} (\enorm{x}^{2 - a} \Laplacian_k - \enorm{x}^a)},
\]
which decomposes as
\[
  \LS{k, a}(z)
  = \sumoplus_{m \in \setN} \id_{\mcalH_k^m(S^{N - 1})} \otimes \LS{k, a}[m](z).
\]
The operator $ \LS{k, a}(z) $ is a Hilbert--Schmidt operator
when $ \RePart z > 0 $, and a unitary operator when $ \RePart z = 0 $.
The operator semigroup $ \faml{\LS{k, a}(z)}{\RePart z \geq 0} $
is introduced and named the \emph{$ (k, a) $-generalized Laguerre semigroup}
by \cite{BKO2012}.

For more details on the $ (k, a) $-generalized Laguerre semigroup,
see \cite[Section~4]{BKO2012}.

\subsection{\texorpdfstring{$ (k, a) $}{(k, a)}-generalized Fourier transform}
\label{ssec:ft}

Let $ k $ be a non-negative multiplicity function, $ a > 0 $, and $ m \in \setN $
with $ \lambda_{k, a, m} = \frac{N - 2 + 2 \dunklindex{k} + 2m}{a} > -1 $.
We define a unitary operator $ \FT{k, a}[m] $
on $ L^2(\setRp, r^{N - 3 + 2 \dunklindex{k} + a} \,dr) $ by
\begin{align*}
  \FT{k, a}[m]
  &= e^{\frac{i\pi}{2} (\lambda_{k, a, 0} + 1)}
    \LS{k, a}[m]\Paren*{\frac{i\pi}{2}} \\
  &= e^{\frac{i\pi}{2} (\lambda_{k, a, 0} + 1)} 
    \exp\Paren*{\frac{i\pi}{2a} (r^{-a} (\vartheta - m)(\vartheta + N - 2 + 2 \dunklindex{k} + m) - r^a)}.
\end{align*}
As a special case of Equation~\cref{eq:diagonalized-ls},
$ \FT{k, a}[m] $ is diagonalized by the orthonormal basis
$ \faml{f_{k, a, m; l}}{l \in \setN} $ and satisfies
\[
  \FT{k, a}[m] f_{k, a, m; l}
  = e^{-i\pi(\frac{m}{a} + l)} f_{k, a, m; l}.
\]

Now, let $ k $ be a non-negative multiplicity function and $ a > 0 $
with $ \lambda_{k, a, 0} = \frac{N - 2 + 2 \dunklindex{k}}{a} > -1 $.
We define a unitary operator $ \FT{k, a} $ on $ L^2(\setR^N, w_{k, a}(x) \,dx) $ by
\[
  \FT{k, a}
  = e^{\frac{i\pi}{2} (\lambda_{k, a, 0} + 1)}
    \LS{k, a}\Paren*{\frac{i\pi}{2}}
  = e^{\frac{i\pi}{2} (\lambda_{k, a, 0} + 1)} 
    \exp\Paren*{\frac{i\pi}{2a} (\enorm{x}^{2 - a} \Laplacian_k - \enorm{x}^a)},
\]
which decomposes as
\[
  \FT{k, a}
  = \sumoplus_{m \in \setN} \id_{\mcalH_k^m(S^{N - 1})} \otimes \FT{k, a}[m].
\]
The unitary operator $ \FT{k, a} $ is introduced and named
the \emph{$ (k, a) $-generalized Fourier transform} by \cite{BKO2012}.

For more details on the $ (k, a) $-generalized Fourier transform,
see \cite[Section~5]{BKO2012}.

\section{\texorpdfstring{$ (k, a) $}{(k, a)}-generalized Fourier transform with negative \texorpdfstring{$ a $}{a}}
\label{sec:negative-a}

\subsection{A unitary transform that intertwines the cases \texorpdfstring{$ a > 0 $}{a > 0} and \texorpdfstring{$ a < 0 $}{a < 0}}

For $ \alpha \in \setR \setminus \setenum{0} $ and $ \beta \in \setR $,
we define a transform $ \kappa_{\alpha, \beta} $ of functions on
$ \setR^N \setminus \setenum{0} $ by
\[
  \kappa_{\alpha, \beta} F(x)
  = \enorm{x}^\beta F(\enorm{x}^{\alpha - 1} x)
  \quad (x \in \setR^N \setminus \setenum{0}).
\]
We also write, for a function $ f $ on $ \setRp $,
\[
  \kappa_{\alpha, \beta} f(r)
  = r^\beta f(r^\alpha)
  \quad (r \in \setRp).
\]

\begin{lemma}\label{thm:kappa}
  For $ \alpha \in \setR \setminus \setenum{0} $ and $ \beta \in \setR $,
  the transform $ \kappa_{\alpha, \beta} $ of functions on $ \setRp $ satisfies
  the following properties.
  \begin{enumarabicp}
    \item The transforms $ \kappa_{\alpha, \beta} $ yield an action of
          the one-dimensional affine group $ \mathrm{Aff}(1, \setR)
          = \set*{\begin{psmallmatrix} \alpha & 1 \\ 0 & \beta \end{psmallmatrix}}%
          {\text{$ \alpha \in \setR \setminus \setenum{0} $, $ \beta \in \setR $}} $,
          namely, $ \kappa_{\alpha, \beta} \circ \kappa_{\alpha', \beta'}
          = \kappa_{\alpha\alpha', \beta + \alpha\beta'} $.
    \item $ \vartheta \circ \kappa_{\alpha, \beta}
          = \kappa_{\alpha, \beta} \circ (\alpha\vartheta + \beta) $,
          where $ \vartheta = r \frac{d}{dr} $ is the Euler operator on $ \setRp $.
    \item For $ d \in \setC $, $ r^{\alpha d} \circ \kappa_{\alpha, \beta}
          = \kappa_{\alpha, \beta} \circ r^d $,
          where $ r^d $ and $ r^{\alpha d} $ denote the multiplication operators.
    \item For $ d \in \setR $, the transform $ \kappa_{\alpha, \beta} $ defines
          a unitary operator from $ L^2(\setRp, r^d \,dr) $ onto
          $ L^2(\setRp, \enorm{\alpha} r^{\alpha d + \alpha - 2\beta - 1} \,dr) $.
  \end{enumarabicp}
\end{lemma}

\begin{proof}
  Confirmed by straightforward calculations.
\end{proof}

Recall from \cref{ssec:l2-theory} that we write $ (\mbfh, \mbfe^+, \mbfe^-) $
for the standard basis for $ \mathfrak{sl}(2, \setR) $.
We define an automorphism $ \tau $ of $ \mathfrak{sl}(2, \setR) $ by
\begin{equation}
  \tau(\mbfh) = \mbfh, \quad
  \tau(\mbfe^+) = -\mbfe^+, \quad
  \tau(\mbfe^-) = -\mbfe^-.
  \label{eq:tau}
\end{equation}
We again write $ \tau $ for its complexification.

Here are our main theorems, which state that the unitary transform
$ \kappa_{-1, -(N - 2 + 2 \dunklindex{k})} $ intertwines the cases $ a > 0 $ and
$ a < 0 $.

\begin{theorem}\label{thm:radial-intertwining}
  Let $ k $ be a (not necessarily non-negative) multiplicity function,
  $ a \in \setC \setminus \setenum{0} $, and $ m \in \setN $.
  The transform $ \kappa_{-1, -(N - 2 + 2 \dunklindex{k})} $ of functions on
  $ \setRp $ satisfies the following properties.
  \begin{enumarabicp}
    \item $ \kappa_{-1, -(N - 2 + 2 \dunklindex{k})}^2 = \id $.
    \item We have the following intertwining relations:
          \begin{align*}
            \kappa_{-1, -(N - 2 + 2 \dunklindex{k})} \circ \DiffH{k, a}[m]
            &= \DiffH{k, -a}[m] \circ \kappa_{-1, -(N - 2 + 2 \dunklindex{k})}, \\
            \kappa_{-1, -(N - 2 + 2 \dunklindex{k})} \circ \DiffEp{k, a}[m]
            &= -\DiffEp{k, -a}[m] \circ \kappa_{-1, -(N - 2 + 2 \dunklindex{k})}, \\
            \kappa_{-1, -(N - 2 + 2 \dunklindex{k})} \circ \DiffEm{k, a}[m]
            &= -\DiffEm{k, -a}[m] \circ \kappa_{-1, -(N - 2 + 2 \dunklindex{k})}.
          \end{align*}
          Thus, for $ X \in \mathfrak{sl}(2, \setC) $, we have
          \[
            \kappa_{-1, -(N - 2 + 2 \dunklindex{k})} \circ \rep{k, a}[m](X)
            = \rep{k, -a}[m](\tau(X)) \circ \kappa_{-1, -(N - 2 + 2 \dunklindex{k})}.
          \]
    \item If $ k $ is non-negative and $ a \in \setR \setminus \setenum{0} $,
          the transform $ \kappa_{-1, -(N - 2 + 2 \dunklindex{k})} $ defines
          a unitary operator from $ L^2(\setRp, r^{N - 3 + 2 \dunklindex{k} + a} \,dr) $
          onto $ L^2(\setRp, r^{N - 3 + 2 \dunklindex{k} - a} \,dr) $.
  \end{enumarabicp}
\end{theorem}

\begin{proof}
  \begin{subproof}{(1)}
    It is a special case of \cref{thm:kappa}~(1).
  \end{subproof}

  \begin{subproof}{(2)}
    By \cref{thm:kappa}~(2) and (3),
    \begin{align*}
      \kappa_{-1, -(N - 2 + 2 \dunklindex{k})} \circ \vartheta
      &= (-\vartheta - (N - 2 + 2 \dunklindex{k})) \circ \kappa_{-1, -(N - 2 + 2 \dunklindex{k})}, \\
      \kappa_{-1, -(N - 2 + 2 \dunklindex{k})} \circ r^d
      &= r^{-d} \circ \kappa_{-1, -(N - 2 + 2 \dunklindex{k})}.
    \end{align*}
    Thus, we have
    \begingroup
      \allowdisplaybreaks
      \begin{align*}
        \kappa_{-1, -(N - 2 + 2 \dunklindex{k})} \circ \DiffH{k, a}[m]
        &= \kappa_{-1, -(N - 2 + 2 \dunklindex{k})} \circ \Paren*{\frac{N - 2 + 2 \dunklindex{k} + a}{a} + \frac{2}{a} \vartheta} \\
        &= \Paren*{\frac{-(N - 2 + 2 \dunklindex{k}) + a}{a} - \frac{2}{a} \vartheta} \circ \kappa_{-1, -(N - 2 + 2 \dunklindex{k})} \\
        &= \DiffH{k, -a}[m] \circ \kappa_{-1, -(N - 2 + 2 \dunklindex{k})}, \\
        \kappa_{-1, -(N - 2 + 2 \dunklindex{k})} \circ \DiffEp{k, a}[m]
        &= \kappa_{-1, -(N - 2 + 2 \dunklindex{k})} \circ \frac{i}{a} r^a \\
        &= \frac{i}{a} r^{-a} \circ \kappa_{-1, -(N - 2 + 2 \dunklindex{k})} \\
        &= -\DiffEp{k, -a}[m] \circ \kappa_{-1, -(N - 2 + 2 \dunklindex{k})}, \\
        \kappa_{-1, -(N - 2 + 2 \dunklindex{k})} \circ \DiffEm{k, a}[m]
        &= \kappa_{-1, -(N - 2 + 2 \dunklindex{k})} \circ \frac{i}{a} r^{-a} (\vartheta - m)(\vartheta + N - 2 + 2 \dunklindex{k} + m) \\
        &= \frac{i}{a} r^{a} (-\vartheta - (N - 2 + 2 \dunklindex{k} + m))(-\vartheta + m) \circ \kappa_{-1, -(N - 2 + 2 \dunklindex{k})} \\
        &= -\DiffEm{k, -a}[m] \circ \kappa_{-1, -(N - 2 + 2 \dunklindex{k})}.
      \end{align*}
    \endgroup
  \end{subproof}

  \begin{subproof}{(3)}
    It is a special case of \cref{thm:kappa}~(4).
  \end{subproof}
\end{proof}

\begin{theorem}\label{thm:total-intertwining}
  Let $ k $ be a (not necessarily non-negative) multiplicity function
  and $ a \in \setC \setminus \setenum{0} $.
  The transform $ \kappa_{-1, -(N - 2 + 2 \dunklindex{k})} $ of functions on
  $ \setR^N \setminus \setenum{0} $ satisfies the following properties.
  \begin{enumarabicp}
    \item $ \kappa_{-1, -(N - 2 + 2 \dunklindex{k})}^2 = \id $.
    \item The transform $ \kappa_{-1, -(N - 2 + 2 \dunklindex{k})} $ commutes
          the translations by the natural action of $ O(N) $
          (in particular, ones by the natural action of the Coxeter group $ \mfrakC $).
    \item The transform $ \kappa_{-1, -(N - 2 + 2 \dunklindex{k})} $ satisfies
          the following intertwining relations:
          \begin{align*}
            \kappa_{-1, -(N - 2 + 2 \dunklindex{k})} \circ \DiffH{k, a}
            &= \DiffH{k, a} \circ \kappa_{-1, -(N - 2 + 2 \dunklindex{k})}, \\
            \kappa_{-1, -(N - 2 + 2 \dunklindex{k})} \circ \DiffEp{k, a}
            &= -\DiffEp{k, -a} \circ \kappa_{-1, -(N - 2 + 2 \dunklindex{k})}, \\
            \kappa_{-1, -(N - 2 + 2 \dunklindex{k})} \circ \DiffEm{k, a}
            &= -\DiffEm{k, -a} \circ \kappa_{-1, -(N - 2 + 2 \dunklindex{k})}.
          \end{align*}
          Thus, for $ X \in \mathfrak{sl}(2, \setC) $, we have
          \[
            \kappa_{-1, -(N - 2 + 2 \dunklindex{k})} \circ \rep{k, a}(X)
            = \rep{k, -a}(\tau(X)) \circ \kappa_{-1, -(N - 2 + 2 \dunklindex{k})}.
          \]
    \item If $ k $ is non-negative and $ a \in \setR \setminus \setenum{0} $,
          the transform $ \kappa_{-1, -(N - 2 + 2 \dunklindex{k})} $ defines
          a unitary operator from $ L^2(\setR^N, w_{k, a}(x) \,dx) $
          onto $ L^2(\setRp, w_{k, a}(x) \,dx) $.
  \end{enumarabicp}
\end{theorem}

\begin{proof}
  The assertion (2) is obvious.
  Since $ \DiffH{k, a}[m] $, $ \DiffEp{k, a}[m] $, and $ \DiffEm{k, a}[m] $
  are the radial parts of $ \DiffH{k, a} $, $ \DiffEp{k, a} $, and $ \DiffEm{k, a} $
  respectively (\cref{thm:radial}) and we have the Hilbert sum decomposition
  \cref{eq:decomposition}, the assertions (1), (3), and (4) follow from
  \cref{thm:radial-intertwining}.
\end{proof}

\begin{remark}
  When $ k = 0 $, the transform in \cref{thm:total-intertwining} is
  \[
    \kappa_{-1, -(N - 2)} F(x)
    = \frac{1}{\enorm{x}^{N - 2}} F\Paren*{\frac{x}{\enorm{x}^2}},
  \]
  which is known as the Kelvin transform~\cite[p.\,39]{Hel2009}.
\end{remark}

\subsection{\texorpdfstring{$ L^2 $}{L2}-theory for the \texorpdfstring{$ \mathfrak{sl}_2 $}{sl2}-triple with negative \texorpdfstring{$ a $}{a}}

Recall from \cref{ssec:l2-theory} that we set
\[
  \lambda_{k, a, m}
  = \frac{N - 2 + 2 \dunklindex{k} + 2m}{a}.
\]
Substituting $ a $ by $ -a $, we get
\[
  \lambda_{k, -a, m}
  = \frac{N - 2 + 2 \dunklindex{k} + 2m}{-a}
  = -\lambda_{k, a, m}.
\]

\begin{proposition}
  Let $ k $ be a non-negative multiplicity function, $ a > 0 $, and $ m \in \setN $.
  For $ l \in \setN $, we define a function $ f_{k, -a, m; l} $ on $ \setRp $ by
  \begin{align*}
    & f_{k, -a, m; l}(r) \\
    &= \kappa_{-1, -(N - 2 + 2 \dunklindex{k})} f_{k, a, m; l}(r) \\
    &= \Paren*{\frac{2^{-\lambda_{k, -a, m} + 1} \Gamma(l + 1)}{a^{-\lambda_{k, -a, m}} \Gamma(-\lambda_{k, -a, m} + l + 1)}}^{1/2}
      r^{-(N - 2 + 2 \dunklindex{k} + m)}
      L^{(-\lambda_{k, -a, m})}_l \Paren*{\frac{2}{a} r^{-a}}
      \exp \Paren*{-\frac{1}{a} r^{-a}}
    \quad (r \in \setRp).
  \end{align*}
  If $ \lambda_{k, -a, m} < 1 $, $ \faml{f_{k, -a, m; l}}{l \in \setN} $ is an
  orthonormal basis for $ L^2(\setRp, r^{N - 3 + 2 \dunklindex{k} - a} \,dr) $.
\end{proposition}

\begin{proof}
  It follows from the fact that $ \faml{f_{k, a, m; l}}{l \in \setN} $
  is an orthonormal basis for $ L^2(\setRp, r^{N - 3 + 2 \dunklindex{k} + a} \,dr) $
  (\cref{thm:onb}) and that
  $ \map{\kappa_{-1, -(N - 2 + 2 \dunklindex{k})}}{L^2(\setRp, r^{N - 3 + 2 \dunklindex{k} + a} \,dr)}%
  {L^2(\setRp, r^{N - 3 + 2 \dunklindex{k} - a} \,dr)} $ is a unitary operator
  (\cref{thm:radial-intertwining}~(3)).
\end{proof}

Recall from \cref{ssec:l2-theory} that we write $ (\mbfn, \mbfk^+, \mbfk^-) $
for the Cayley transform of the standard basis $ (\mbfh, \mbfe^+, \mbfe^-) $
for $ \mathfrak{sl}(2, \setR) $. Note that the automorphism $ \tau $ defined by
Equation~\cref{eq:tau} satisfies
\[
  \tau(\mbfk) = -\mbfk, \quad
  \tau(\mbfn^+) = -\mbfn^-, \quad
  \tau(\mbfn^-) = -\mbfn^+.
\]

\begin{proposition}
  Let $ k $ be a non-negative multiplicity function, $ a > 0 $, and $ m \in \setN $
  with $ \lambda_{k, -a, m} = \frac{N - 2 + 2 \dunklindex{k} + 2m}{-a} < 1 $.
  For $ l \in \setN $, we have
  \begin{align*}
    \rep{k, -a}[m](\mbfk)   f_{k, -a, m; l} &= (\lambda_{k, -a, m} - 2l - 1) f_{k, -a, m; l}, \\
    \rep{k, -a}[m](\mbfn^+) f_{k, -a, m; l} &= -i \sqrt{l (-\lambda_{k, -a, m} + l)} \, f_{k, -a, m; l - 1}, \\
    \rep{k, -a}[m](\mbfn^-) f_{k, -a, m; l} &= -i \sqrt{(l + 1)(-\lambda_{k, -a, m} + l + 1)} \, f_{k, -a, m; l + 1},
  \end{align*}
  where we regard $ f_{k, -a, m; -1} = 0 $.
\end{proposition}

\begin{proof}
  Apply $ \kappa_{-1, -(N - 2 + 2 \dunklindex{k})} $ to the both sides
  of the equations in \cref{thm:action},
  and then use \cref{thm:radial-intertwining}~(2).
\end{proof}

In the same way as \cref{thm:radial-spectrum,thm:total-spectrum},
we obtain the following spectral properties.

\begin{corollary}\label{thm:radial-spectrum-with-negative-a}
  Let $ k $ be a non-negative multiplicity function, $ a > 0 $, and $ m \in \setN $
  with $ \lambda_{k, -a, m} = \frac{N - 2 + 2 \dunklindex{k} + 2m}{-a} < 1 $.
  The differential-difference operator
  \[
    \rep{k, -a}[m](\mbfk)
    = -\frac{1}{a} (r^{-a} - r^a (\vartheta - m) (\vartheta + N - 2 + 2 \dunklindex{k} + m))
  \]
  restricted to $ W_{k, -a}^{(m)} = \lspan_{\setC} \set{f_{k, -a, m; l}}{l \in \setN} $
  is essentially self-adjoint on $ L^2(\setRp, r^{N - 3 + 2 \dunklindex{k} - a} \,dr) $
  and diagonalized by the orthonormal basis $ \faml{f_{k, -a, m; l}}{l \in \setN} $.
  The discrete spectrum corresponding to $ f_{k, -a, m; l} $ is
  $ \lambda_{k, -a, m} - 2l - 1 $.
  \qed
\end{corollary}

\begin{corollary}
  Let $ k $ be a non-negative multiplicity function and $ a > 0 $
  with $ \lambda_{k, -a, 0} = \frac{N - 2 + 2 \dunklindex{k}}{-a} < 1 $.
  The differential-difference operator
  \[
    \rep{k, -a}(\mbfk)
    = \frac{1}{-a} (\enorm{x}^{-a} - \enorm{x}^{2 + a} \Laplacian_k)
  \]
  restricted to $ W_{k, -a} = \bigoplus_{m \in \setN} \mcalH_k^m(S^{N - 1}) \otimes W_{k, -a}^{(m)} $
  is essentially self-adjoint on $ L^2(\setR^N, w_{k, -a}(x) \,dx) $,
  has no continuous spectra, and has the set of discrete spectra
  \[
    \begin{cases}
      \set{\lambda_{k, -a, m} - 2l - 1}{\text{$ m $, $ l \in \setN $}} & (N \geq 2) \\
      \set{\frac{2 \dunklindex{k} \pm 1}{-a} - 2l - 1}{l \in \setN}    & (N = 1).
    \end{cases}
    \pushQED{\qed}\qedhere\popQED
  \]
\end{corollary}

\subsection{Lifting to a unitary representation \texorpdfstring{$ \Rep{k, a} $}{Ωk, a} with negative \texorpdfstring{$ a $}{a}}

Recall from \cref{ssec:lifting} that we write $ \widetilde{\mathit{SL}}(2, \setR) $
for the universal covering Lie group of $ \mathit{SL}(2, \setR) $,
and $ \widetilde{\mathit{SO}}(2) $ for the connected Lie subgroup of
$ \widetilde{\mathit{SL}}(2, \setR) $ with Lie algebra
$ \mathfrak{so}(2) = \setR i\mbfk $.

In \cref{ssec:lifting}, we use a lowest weight
$ (\mathfrak{sl}(2, \setC), \widetilde{\mathit{SO}}(2)) $-module
$ \pi_{\widetilde{\mathit{SO}}(2)}(\lambda) $ of weight $ \lambda + 1 $,
and its lifting $ \pi(\lambda) $. Correspondingly, in this subsection,
we use a highest weight $ (\mathfrak{sl}(2, \setC), \widetilde{\mathit{SO}}(2)) $-module
and its lifting.

Let $ (\varpi, V) $ be a $ (\mathfrak{sl}(2, \setC), \widetilde{\mathit{SO}}(2)) $-module.
We say that $ v \in V \setminus \setenum{0} $ is a \emph{highest weight vector
of weight $ \mu \in \setC $} if
\[
  \varpi(\mbfk) v = \mu v
  \quad \text{and} \quad
  \varpi(\mbfn^+) v = 0.
\]
We say that $ (\varpi, V) $ is a \emph{highest weight module of weight $ \mu $}
if $ V $ is generated by such $ v $.
As with the case of lowest weight modules, for each $ \lambda \in \setC $,
there exists a unique irreducible highest weight module of weight $ \lambda - 1 $
(up to isomorphism), for which we write $ \pi'_{\widetilde{\mathit{SO}}(2)}(\lambda) $.
Moreover, for $ \lambda \leq 1 $, $ \pi'_{\widetilde{\mathit{SO}}(2)}(\lambda) $
lifts to a unique irreducible unitary representation $ \pi'(\lambda) $
of $ \widetilde{\mathit{SL}}(2, \setR) $,
which is infinite-dimensional except for the one-dimensional trivial representation
$ \pi'(1) $.

\begin{remark}\label{rem:lowest-and-highest-weight-module}
  With the automorphism $ \tau $ of $ \mathfrak{sl}(2, \setR) $ defined by
  \cref{eq:tau} and its lifting $ \widetilde{\tau} $ to an automorphism
  of $ \widetilde{\mathit{SL}}(2, \setR) $, we have
  \begin{align*}
    \pi_{\widetilde{\mathit{SO}}(2)}(\lambda) \circ \tau
    &= \pi'_{\widetilde{\mathit{SO}}(2)}(-\lambda), \\
    \pi(\lambda) \circ \widetilde{\tau}
    &= \pi'(-\lambda).
  \end{align*}
\end{remark}

\begin{theorem}
  Let $ k $ be a non-negative multiplicity function, $ a > 0 $, and $ m \in \setN $
  with $ \lambda_{k, -a, m} = \frac{N - 2 + 2 \dunklindex{k} + 2m}{-a} < 1 $.
  We set
  \[
    W_{k, -a}^{(m)}
    = \lspan_{\setC} \set{f_{k, -a, m; l}}{l \in \setN}.
  \]
  Then, $ (\rep{k, -a}[m], W_{k, -a}^{(m)}) $ equips a natural
  $ (\mathfrak{sl}(2, \setC), \widetilde{\mathit{SO}}(2)) $-module structure,
  and is isomorphic to $ \pi'_{\widetilde{\mathit{SO}}(2)}(\lambda_{k, -a, m}) $.
  Moreover, the $ (\mathfrak{sl}(2, \setC), \widetilde{\mathit{SO}}(2)) $-module
  $ (\rep{k, -a}[m], W_{k, -a}^{(m)}) $ lifts to a unique unitary representation
  $ \Rep{k, -a}[m] $ of $ \widetilde{\mathit{SL}}(2, \setR) $
  on the Hilbert space $ L^2(\setRp, r^{N - 3 + 2 \dunklindex{k} - a} \,dr) $,
  which is unitarily equivalent to $ \pi'(\lambda_{k, -a, m}) $.
\end{theorem}

\begin{proof}
  It follows from the result for positive $ a $ (\cref{thm:radial-lifting}),
  the intertwining property of $ \kappa_{-1, -(N - 2 + 2 \dunklindex{k})} $
  (\cref{thm:radial-intertwining}),
  and \cref{rem:lowest-and-highest-weight-module}.
\end{proof}

\begin{theorem}
  Let $ k $ be a non-negative multiplicity function and $ a > 0 $
  with $ \lambda_{k, -a, 0} = \frac{N - 2 + 2 \dunklindex{k}}{-a} < 1 $.
  We set
  \[
    W_{k, -a}
    = \bigoplus_{m \in \setN} \mcalH_k^m(S^{N - 1}) \otimes W_{k, -a}^{(m)}.
  \]
  Then, $ (\rep{k, -a}, W_{k, -a}) $ equips a natural
  $ \mfrakC \times (\mathfrak{sl}(2, \setC), \widetilde{\mathit{SO}}(2)) $-module
  structure, where $ \mfrakC $ denotes the Coxeter group.
  Moreover, the $ \mfrakC \times (\mathfrak{sl}(2, \setC), \widetilde{\mathit{SO}}(2)) $-module
  $ (\rep{k, -a}, W_{k, -a}) $ lifts to a unique unitary representation
  $ \Rep{k, -a} $ of $ \mfrakC \times \widetilde{\mathit{SL}}(2, \setR) $
  on the Hilbert space $ L^2(\setR \setminus \setenum{0}, w_{k, -a}(x) \,dx) $,
  which decomposes as
  \[
    L^2(\setR^N, w_{k, -a}(x) \,dx)
    = \sumoplus_{m \in \setN} \mcalH_k^m(S^{N - 1})
      \otimes L^2(\setRp, r^{N - 3 + 2 \dunklindex{k} - a} \,dr).
  \]
  Here, $ \mfrakC $ acts on $ \mcalH_k^m(S^{N - 1}) $,
  and $ \widetilde{\mathit{SL}}(2, \setR) $ acts on
  $ L^2(\setRp, r^{N - 3 + 2 \dunklindex{k} - a} \,dr) $
  via the unitary representation $ \Rep{k, -a}[m] $.
\end{theorem}

\begin{proof}
  It follows from the result for positive $ a $ (\cref{thm:total-lifting})
  and the intertwining property of $ \kappa_{-1, -(N - 2 + 2 \dunklindex{k})} $
  (\cref{thm:total-intertwining}).
\end{proof}

\subsection{\texorpdfstring{$ (k, a) $}{(k, a)}-generalized Laguerre semigroup with negative \texorpdfstring{$ a $}{a}}
\label{ssec:ls-with-negative-a}

In this subsection, we extend the definition of the $ (k, a) $-generalized
Laguerre semigroup to the case $ a < 0 $.

Let $ k $ be a non-negative multiplicity function, $ a > 0 $, and $ m \in \setN $
with $ \lambda_{k, -a, m} = \frac{N - 2 + 2 \dunklindex{k} + 2m}{-a} < 1 $.
For $ z \in \setC $, we define
\[
  \LS{k, -a}[m](z)
  = \exp(-z \rep{k, -a}[m](\mbfk))
  = \exp\Paren*{\frac{z}{-a} (r^a (\vartheta - m)(\vartheta + N - 2 + 2 \dunklindex{k} + m) - r^{-a})}.
\]
By \cref{thm:radial-spectrum-with-negative-a}, each operator $ \LS{k, -a}[m](z) $
on $ L^2(\setRp, r^{N - 3 + 2 \dunklindex{k} - a} \,dr) $
is diagonalized by the orthonormal basis $ \faml{f_{k, -a, m; l}}{l \in \setN} $
and satisfies
\begin{equation}
  \LS{k, -a}[m](z) f_{k, -a, m; l}
  = e^{-z (\lambda_{k, -a, m} - 2l - 1)} f_{k, -a, m; l}.
  \label{eq:diagonalized-ls-with-negative-a}
\end{equation}
In particular, $ \LS{k, -a}[m](z) $ is a Hilbert--Schmidt operator
when $ \RePart z < 0 $, and a unitary operator when $ \RePart z = 0 $.

Now, let $ k $ be a non-negative multiplicity function and $ a > 0 $
with $ \lambda_{k, -a, 0} = \frac{N - 2 + 2 \dunklindex{k}}{-a} < 1 $.
For $ z \in \setC $, we define
\[
  \LS{k, -a}(z)
  = \exp(-z \rep{k, -a}(\mbfk))
  = \exp\Paren*{\frac{z}{-a} (\enorm{x}^{2 + a} \Laplacian_k - \enorm{x}^{-a})},
\]
which decomposes as
\[
  \LS{k, -a}(z)
  = \sumoplus_{m \in \setN} \id_{\mcalH_k^m(S^{N - 1})} \otimes \LS{k, -a}[m](z).
\]
The operator $ \LS{k, -a}(z) $ is a Hilbert--Schmidt operator
when $ \RePart z < 0 $, and a unitary operator when $ \RePart z = 0 $.
We call $ \faml{\LS{k, -a}(z)}{\RePart z \leq 0} $
the \emph{$ (k, -a) $-generalized Laguerre semigroup}.

\begin{remark}\label{rem:ls-and-kappa}
  By the intertwining property of $ \kappa_{-1, -(N - 2 + 2 \dunklindex{k})} $
  (\cref{thm:radial-intertwining,thm:total-intertwining}), we have
  \begin{align*}
    \kappa_{-1, -(N - 2 + 2 \dunklindex{k})} \circ \LS{k, a}[m](z)
    &= \LS{k, -a}[m](-z) \circ \kappa_{-1, -(N - 2 + 2 \dunklindex{k})}, \\
    \kappa_{-1, -(N - 2 + 2 \dunklindex{k})} \circ \LS{k, a}(z)
    &= \LS{k, -a}(-z) \circ \kappa_{-1, -(N - 2 + 2 \dunklindex{k})}.
  \end{align*}
\end{remark}

\subsection{\texorpdfstring{$ (k, a) $}{(k, a)}-generalized Fourier transform with negative \texorpdfstring{$ a $}{a}}
\label{ssec:ft-with-negative-a}

In this subsection, we extend the definition of the $ (k, a) $-generalized
Fourier transform to the case $ a < 0 $.

Let $ k $ be a non-negative multiplicity function, $ a > 0 $, and $ m \in \setN $
with $ \lambda_{k, -a, m} = \frac{N - 2 + 2 \dunklindex{k} + 2m}{-a} < 1 $.
We define a unitary operator $ \FT{k, -a}[m] $
on $ L^2(\setRp, r^{N - 3 + 2 \dunklindex{k} - a} \,dr) $ by
\begin{align*}
  \FT{k, -a}[m]
  &= e^{\frac{i\pi}{2} (\lambda_{k, -a, 0} + 1)}
    \LS{k, -a}[m]\Paren*{\frac{i\pi}{2}} \\
  &= e^{\frac{i\pi}{2} (\lambda_{k, -a, 0} + 1)} 
    \exp\Paren*{\frac{i\pi}{-2a} (r^a (\vartheta - m)(\vartheta + N - 2 + 2 \dunklindex{k} + m) - r^{-a})}.
\end{align*}
As a special case of Equation~\cref{eq:diagonalized-ls-with-negative-a},
$ \FT{k, -a}[m] $ is diagonalized by the orthonormal basis
$ \faml{f_{k, -a, m; l}}{l \in \setN} $ and satisfies
\[
  \FT{k, -a}[m] f_{k, -a, m; l}
  = -e^{i\pi(\frac{m}{a} + l)} f_{k, -a, m; l}.
\]

Now, let $ k $ be a non-negative multiplicity function and $ a > 0 $
with $ \lambda_{k, -a, 0} = \frac{N - 2 + 2 \dunklindex{k}}{-a} < 1 $.
We define a unitary operator $ \FT{k, -a} $ on $ L^2(\setR^N, w_{k, -a}(x) \,dx) $ by
\[
  \FT{k, -a}
  = e^{\frac{i\pi}{2} (\lambda_{k, -a, 0} + 1)}
    \LS{k, -a}\Paren*{\frac{i\pi}{2}}
  = e^{\frac{i\pi}{2} (\lambda_{k, -a, 0} + 1)} 
    \exp\Paren*{\frac{i\pi}{-2a} (\enorm{x}^{2 + a} \Laplacian_k - \enorm{x}^{-a})},
\]
which decomposes as
\[
  \FT{k, -a}
  = \sumoplus_{m \in \setN} \id_{\mcalH_k^m(S^{N - 1})} \otimes \FT{k, -a}[m].
\]
We call the unitary operator $ \FT{k, -a} $
the \emph{$ (k, -a) $-generalized Fourier transform}.

\begin{remark}
  As a special case of \cref{rem:ls-and-kappa}, we have
  \begin{align}
    \kappa_{-1, -(N - 2 + 2 \dunklindex{k})} \circ \FT{k, a}[m]
    &= -(\FT{k, -a}[m])^{-1} \circ \kappa_{-1, -(N - 2 + 2 \dunklindex{k})},
      \notag \\
    \kappa_{-1, -(N - 2 + 2 \dunklindex{k})} \circ \FT{k, a}
    &= -\FT{k, -a}^{-1} \circ \kappa_{-1, -(N - 2 + 2 \dunklindex{k})}.
      \label{eq:ft-and-kappa}
  \end{align}
\end{remark}

\begin{remark}
  De Bie--Ørsted--Somberg--Souček~\cite{DOSS2012}, following the idea of\linebreak{}
  \cite{BKO2009}, considered a ``$ (k, a) $-deformation'' $ \mscrD_{k, a} $
  of the Dirac operator (Equation~(3.2) in the paper) and found that
  $ \kappa_{-1, -(N - 2 + 2 \dunklindex{k})} $ ($ \mbfI_k $ in their notation)
  intertwines $ \mscrD_{k, 2} $ and $ \mscrD_{k, -2} $
  (Proposition~4 in the paper). They deduced from this that
  $ \kappa_{-1, -(N - 2 + 2 \dunklindex{k})} $ intertwines $ \FT{k, 2} $ and
  $ \FT{k, -2} $, which corresponds to the case $ a = 2 $ in
  Equation~\cref{eq:ft-and-kappa}.
  On the other hand, in the present paper, we proved that
  $ \kappa_{-1, -(N - 2 + 2 \dunklindex{k})} $ intertwines the
  $ \mathfrak{sl}_2 $-triples $ (\DiffH{k, a}, \DiffEp{k, a}, \DiffEm{k, a}) $
  and $ (\DiffH{k, -a}, \DiffEp{k, -a}, \DiffEm{k, -a}) $ for $ a \neq 0 $
  (\cref{thm:total-intertwining}), which was not treated in \cite{DOSS2012}.
  Our approach to obtain Equation~\cref{eq:ft-and-kappa} is to use this
  intertwining property.
\end{remark}

\section*{Acknowledgements}

The author is deeply grateful to his supervisor Professor Toshiyuki Kobayashi
for his kind and patient guidance. His many insightful suggestions served as
a guiding light for the author's research. It is an honor to dedicate
the present paper to him.
The author also thanks Kazuki Kannaka, whose comments led the author
to consider $ (k, a) $-generalized Fourier transform with negative $ a $.

The author warmly thanks Professor Ali Baklouti, his colleagues,
and Professor Hideyuki Ishi for their kindness during
``7th Tunisian-Japanese Conference: Geometric and Harmonic Analysis
on Homogeneous Spaces and Applications in Honor of Professor Toshiyuki Kobayashi''.

This work was supported by World-leading Innovative Graduate Study for Frontiers
of Mathematical Sciences and Physics (WINGS-FMSP).

\end{document}